\documentclass[12p]{amsart}
\usepackage{amssymb}
\usepackage{amsmath}
\usepackage{amsfonts}
\usepackage{geometry}
\usepackage{graphicx}
\usepackage{mathrsfs,amssymb,latexsym}

\usepackage{hyperref}
\usepackage{cleveref}

\theoremstyle{plain}

\newtheorem{lemma}{Lemma}

\newtheorem{proposition}{Proposition}
\newtheorem{remark}{Remark}

\newtheorem{theorem}{Theorem}
\numberwithin{equation}{section}

\usepackage{cancel}

\begin{document}
\title{Global well-posedness of the radial conformal nonlinear wave equation with initial data in a critical space}
\date{\today}
\author{Benjamin Dodson}
\maketitle

\begin{abstract}
In this note we prove global well-posedness and scattering for the conformal, defocusing, nonlinear wave equation with radial initial data in a critical Besov space. We also prove a polynomial bound on the scattering norm.
\end{abstract}

\section{Introduction}
Consider the defocusing wave equation
\begin{equation}\label{1.1}
u_{tt} - \Delta u + |u|^{\frac{4}{d - 1}} u = 0, \qquad u(0,x) = u_{0}, \qquad u_{t}(0,x) = u_{1},
\end{equation}
where $(u_{0}, u_{1}) \in B_{1,1}^{\frac{d}{2} + \frac{1}{2}} \times B_{1,1}^{\frac{d}{2} - \frac{1}{2}}$ is radially symmetric, in any dimension $d \geq 3$. Previous work on this problem has addressed the $d = 3$ and $d = 5$ cases. In this paper we prove

\begin{theorem}\label{t1.1}
When $d \geq 4$, if $(u_{0}, u_{1}) \in B_{1,1}^{\frac{d}{2} + \frac{1}{2}} \times B_{1,1}^{\frac{d}{2} - \frac{1}{2}}$ is radially symmetric, then $(\ref{1.1})$ has a global solution that scatters. Moreover, if
\begin{equation}\label{1.2}
\| u_{0} \|_{B_{1,1}^{\frac{d}{2} + \frac{1}{2}}} + \| u_{1} \|_{B_{1,1}^{\frac{d}{2} - \frac{1}{2}}} = A,
\end{equation}
then
\begin{equation}\label{1.3}
\| u \|_{L_{t,x}^{\frac{2(d + 1)}{d - 1}}(\mathbb{R} \times \mathbb{R}^{d})} \lesssim  A + A^{\frac{d - 1}{d - 3}}.
\end{equation}
\end{theorem}

In general, the nonlinear wave equation,
\begin{equation}\label{1.4}
u_{tt} - \Delta u + |u|^{p - 1} u = 0,
\end{equation}
has the scaling symmetry
\begin{equation}\label{1.5}
u(t,x) \mapsto \lambda^{\frac{2}{p - 1}} u(\lambda t, \lambda x).
\end{equation}
That is, if $u(t,x)$ solves $(\ref{1.4})$, then the right hand side of $(\ref{1.5})$ solves $(\ref{1.4})$ for any $\lambda > 0$. Then set
\begin{equation}\label{1.6}
s_{c} = \frac{d}{2} - \frac{2}{p - 1}.
\end{equation}
Then the $\dot{H}^{s_{c}}(\mathbb{R}^{d}) \times \dot{H}^{s_{c} - 1}(\mathbb{R}^{d})$ norm of the initial data is invariant under $(\ref{1.5})$. So for $p - 1 = \frac{4}{d - 1}$, $s_{c} = \frac{1}{2}$. The results of \cite{lindblad1995existence} completely determine the local behavior of $(\ref{1.1})$. Moreover, global well-posedness and scattering hold for small initial data. See Theorem $\ref{t2.2}$ (where we recall the work of \cite{lindblad1995existence}). Moreover, \cite{lindblad1995existence} proved that this result is sharp, that is, local well-posedness for $(\ref{1.1})$ does not hold for initial data in $\dot{H}^{s} \times \dot{H}^{s - 1}$ for any $s < \frac{1}{2}$.
\begin{remark}
The $L_{t,x}^{\frac{2(d + 1)}{d - 1}}(\mathbb{R} \times \mathbb{R}^{d})$ norm is invariant under the scaling $(\ref{1.5})$.
\end{remark}

It is conjectured that global well-posedness and scattering hold for $(\ref{1.1})$ for initial data in $\dot{H}^{1/2} \times \dot{H}^{-1/2}$. This has been resolved in the affirmative when $d = 3$ for radially symmetric initial data. See \cite{dodson2018global} and also \cite{dodson2018global2}.\medskip

The Besov spaces $B_{1,1}^{s}$ are defined by the Littlewood--Paley partition of unity.
\begin{equation}\label{1.7}
\| f \|_{B_{1,1}^{s}(\mathbb{R}^{d})} = \sum_{j \in \mathbb{Z}} 2^{js} \| P_{j} f \|_{L^{1}(\mathbb{R}^{d})}.
\end{equation}
The Sobolev embedding theorem implies $B_{1,1}^{\frac{d}{2} + \frac{1}{2}} \subset \dot{H}^{\frac{1}{2}}$ and $B_{1,1}^{\frac{d}{2} - \frac{1}{2}} \subset \dot{H}^{-1/2}$. The space $B_{1,1}^{\frac{d}{2} + \frac{1}{2}} \times B_{1,1}^{\frac{d}{2} - \frac{1}{2}}$ for the initial data is also invariant under $(\ref{1.5})$.\medskip

Study of dispersive partial differential equations with initial data in a Besov space has proved to be quite fruitful. 
\begin{proposition}\label{p1.2}
When $d = 3$, for every radial initial data $(u_{0}, u_{1}) \in B_{1,1}^{2} \times B_{1,1}^{1}(\mathbb{R}^{3})$, let $u$ be the solution to $(\ref{1.1})$. Then there exists a function $A : [0, \infty) \rightarrow [0, \infty)$ such that
\begin{equation}\label{1.8}
\| u \|_{L_{t,x}^{4}(\mathbb{R} \times \mathbb{R}^{3})} \leq A(\| (u_{0}, u_{1}) \|_{B_{1,1}^{2} \times B_{1,1}^{1}(\mathbb{R}^{3})}).
\end{equation}
When $d = 5$, for every radial initial data $(u_{0}, u_{1}) \in B_{1,1}^{3} \times B_{1,1}^{2}(\mathbb{R}^{5})$, let $u$ be the solution to $(\ref{1.1})$. Then there exists a function $A : [0, \infty) \rightarrow [0, \infty)$ and a parameter $\delta_{1} > 0$ that depends on the initial data $(u_{0}, u_{1})$ such that
\begin{equation}\label{1.9}
\| u \|_{L_{t,x}^{3}(\mathbb{R} \times \mathbb{R}^{5})} \leq A(\| (u_{0}, u_{1}) \|_{B_{1,1}^{3} \times B_{1,1}^{2}}, \delta_{1}).
\end{equation}
\end{proposition}
\begin{proof}
See \cite{dodson2018globalAPDE} for the result in dimension $d = 3$ and \cite{miao2020global} in dimension $d = 5$.
\end{proof}
In this paper we extend Proposition $\ref{p1.2}$ to any dimension and in addition obtain polynomial bounds in dimensions $d \geq 4$. We do not obtain an improved bound in dimension $d = 3$.\medskip

Polynomial bounds were also obtained for the nonlinear Schr{\"o}dinger equation with initial data in a critical Besov space in \cite{dodson2021scattering}. As in the nonlinear wave case, polynomial bounds were obtained for solutions to the nonlinear Schr{\"o}dinger equation
\begin{equation}\label{1.10}
i u_{t} + \Delta u = |u|^{p - 1} u,
\end{equation}
when $1 < p < 3$, but when $p = 3$ only bounds of the form $(\ref{1.8})$ were obtained.\medskip

There are many initial data that fall into the category of radially symmetric $B_{1,1}^{\frac{d}{2} + \frac{1}{2}} \times B_{1,1}^{\frac{d}{2} - \frac{1}{2}}$. Suppose for example that $u_{0}$ is a Gaussian function and $u_{1} = 0$. Since the Gaussian function is smooth, rapidly decreasing, as are all its derivatives, the initial data satisfies $(\ref{3.2})$, and therefore the work of \cite{strauss1968decay} implies that the solution to $(\ref{1.1})$ with $u_{0} = C e^{-|x|^{2}}$ and $u_{1} = 0$ satisfies
\begin{equation}\label{1.11}
\| u \|_{L_{t,x}^{\frac{2(d + 1)}{d - 1}}(\mathbb{R} \times \mathbb{R}^{d})} \lesssim C + C^{c(d)},
\end{equation}
where $c(d) > 1$ depends only on the dimension.
\begin{remark}
We will recall the work of \cite{strauss1968decay} in Lemma $\ref{l3.2}$, 
\end{remark}

 The same bounds also hold for initial data rescaled under $(\ref{1.5})$. Thus, $(\ref{1.11})$ also holds for $u_{1} = 0$ and $u_{0} = C \lambda^{\frac{d - 1}{2}} e^{-\lambda^{2} |x|^{2}}$, for any $\lambda > 0$. Moreover, the solution to $(\ref{1.1})$ with such initial data also satisfies $(\ref{1.11})$.\medskip

Now take $\lambda_{1} \ll 1$ and $\lambda_{2} \gg 1$, and assume $C_{1}, C_{2}, C_{0} > 0$ are constants. Then take $u_{1} = 0$ and
\begin{equation}\label{1.12}
u_{0} = C_{1} \lambda_{1}^{\frac{d - 1}{2}} e^{-\lambda_{1}^{2} |x|^{2}} + C_{0} e^{-|x|^{2}} + C_{2} \lambda_{2}^{\frac{d - 1}{2}} e^{-\lambda_{2}^{2} |x|^{2}}.
\end{equation}
In this case, it is reasonable to conjecture that the bounds on a solution to $(\ref{1.1})$ will be some polynomial function of $C_{1} + C_{2} + C_{0}$, by $(\ref{1.11})$ and the heuristic that solutions to $(\ref{1.1})$ at very different frequencies are basically decoupled. However, if one merely plugs $(\ref{1.12})$ into Lemma $\ref{l3.2}$, in the case of initial data given by $(\ref{1.12})$, the scattering size bounds on the right hand side of $(\ref{1.11})$ will depend on $\lambda_{1}^{-1}$ and $\lambda_{2}$. However taking $\lambda_{1} \searrow 0$ and $\lambda_{2} \nearrow \infty$ ought to increase the strength of the decoupling.\medskip

By proving Theorem $\ref{t1.1}$ we resolve this issue.
\begin{remark}
Throughout this paper, a solution to $(\ref{1.1})$ refers to a strong solution. That is, $u$ satisfies Duhamel's principle,
\begin{equation}\label{1.13}
(u(t), u_{t}(t)) = S(t)(u_{0}, u_{1}) - \int_{0}^{t} S(t - \tau)(0, |u|^{\frac{4}{p - 1}} u) d\tau.
\end{equation}
The notation $S(t)(f, g)$ denotes $(u(t), u_{t}(t))$, where $u$ is a solution to the linear wave equation,
\begin{equation}\label{1.14}
u_{tt} - \Delta u = 0, \qquad u(0,x) = f(x), \qquad u_{t}(0,x) = g(x).
\end{equation}
Global well-posedness has the canonical definition, taking a solution to mean a strong solution.\medskip

Scattering implies that as $t \rightarrow +\infty$ or $t \rightarrow -\infty$, the solution looks like a solution to $(\ref{1.14})$. That is, there exist $(u_{0}^{+}, u_{1}^{+}) \in \dot{H}^{1/2} \times \dot{H}^{-1/2}$ and $(u_{0}^{-}, u_{1}^{-}) \in \dot{H}^{1/2} \times \dot{H}^{-1/2}$ such that
\begin{equation}\label{1.15}
\lim_{t \rightarrow \infty} \| S(-t)(u(t), u_{t}(t)) - (u_{0}^{+}, u_{1}^{+}) \|_{\dot{H}^{1/2} \times \dot{H}^{-1/2}}, \qquad \lim_{t \rightarrow -\infty} \| S(-t)(u(t), u_{t}(t)) - (u_{0}^{-}, u_{1}^{-}) \|_{\dot{H}^{1/2} \times \dot{H}^{-1/2}} = 0.
\end{equation}
\end{remark}

\section{Strichartz estimates and local well-posedness}
Global well-posedness and scattering for $(\ref{1.1})$ is equivalent to obtaining the bound
\begin{equation}\label{2.1}
\| u \|_{L_{t,x}^{\frac{2(d + 1)}{d - 1}}(\mathbb{R} \times \mathbb{R}^{d})} < \infty.
\end{equation}

\begin{proof}[Proof of equivalence]
This fact follows from the Strichartz estimates.
\begin{theorem}[Strichartz estimates for the wave equation]\label{t2.1}
Let $I$ be a time interval and let $u : I \times \mathbb{R}^{d} \rightarrow \mathbb{C}$ be a Schwartz solution to the wave equation
\begin{equation}\label{2.2}
u_{tt} - \Delta u = F, \qquad u(t_{0}) = u_{0}, \qquad u_{t}(t_{0}) = u_{1},
\end{equation}
for some $t_{0} \in I$. Then
\begin{equation}\label{2.3}
\| u \|_{L_{t}^{q} L_{x}^{r}(I \times \mathbb{R}^{d})} + \| u \|_{C_{t}^{0} \dot{H}_{x}^{s}(I \times \mathbb{R}^{d})} + \| \partial_{t} u \|_{C_{t}^{0} \dot{H}_{x}^{s - 1}(I \times \mathbb{R}^{d})} \lesssim_{q, r, s, d} \| u_{0} \|_{\dot{H}_{x}^{s}(\mathbb{R}^{d})} + \| u_{1} \|_{\dot{H}_{x}^{s - 1}(\mathbb{R}^{d})} + \| F \|_{L_{t}^{\tilde{q}'} L_{x}^{\tilde{r}'}(I \times \mathbb{R}^{d})},
\end{equation}
whenever $s \geq 0$, $2 \leq q, \tilde{q} \leq \infty$ and $2 \leq r, \tilde{r} < \infty$ obey the scaling condition
\begin{equation}\label{2.4}
\frac{1}{q} + \frac{d}{r} = \frac{d}{2} - s = \frac{1}{\tilde{q}'} + \frac{d}{\tilde{r}'} - 2,
\end{equation}
and the wave admissibility conditions
\begin{equation}\label{2.5}
\frac{1}{q} + \frac{d - 1}{2r}, \frac{1}{\tilde{q}} + \frac{d - 1}{2\tilde{r}} \leq \frac{d - 1}{4}.
\end{equation}
\end{theorem}
\begin{proof}
This theorem is copied from \cite{tao2006nonlinear}. The proof of this theorem is found in \cite{kato1994lq}, \cite{ginibre1995generalized}, \cite{kapitanski1989some}, \cite{lindblad1995existence}, \cite{sogge1995lectures}, \cite{shatah2000geometric}, \cite{keel1998endpoint}.
\end{proof}

When $s = \frac{1}{2}$, $(q, r) = (\frac{2(d + 1)}{d - 1}, \frac{2(d + 1)}{d - 1})$ satisfies $(\ref{2.4})$ and $(\ref{2.5})$, as does $(\tilde{q}', \tilde{r}') = (\frac{2(d + 1)}{d + 3}, \frac{2(d + 1)}{d + 3})$. Furthermore, by H{\"o}lder's inequality,
\begin{equation}\label{2.6}
\| |u|^{\frac{4}{d - 1}} u \|_{L_{t,x}^{\frac{2(d + 1)}{d + 3}}(\mathbb{R} \times \mathbb{R}^{d})} \leq \| u \|_{L_{t,x}^{\frac{2(d + 1)}{d - 1}}(\mathbb{R} \times \mathbb{R}^{d})}^{1 + \frac{4}{d - 1}}.
\end{equation}
Therefore, standard small data arguments and the Lebesgue dominated convergence theorem imply
\begin{theorem}\label{t2.2}
The nonlinear wave equation $(\ref{1.1})$ is globally well-posed and scattering for initial data sufficiently small, that is, $\| u_{0} \|_{\dot{H}^{1/2}} + \| u_{1} \|_{\dot{H}^{-1/2}} < \epsilon_{0}(d)$, where $0 < \epsilon_{0}(d) \ll 1$. For any $(u_{0}, u_{1}) \in \dot{H}^{1/2} \times \dot{H}^{-1/2}$, there exists some $T(u_{0}, u_{1}) > 0$ such that $(\ref{1.1})$ is locally well-posed on an interval $(-T, T)$.
\end{theorem}
\begin{proof}
This theorem was proved in \cite{lindblad1995existence}.
\end{proof}

Now suppose $u$ has a solution to $(\ref{1.1})$ on $[0, T)$ that satisfies
\begin{equation}\label{2.7}
\| u \|_{L_{t,x}^{\frac{2(d + 1)}{d - 1}}([0, T) \times \mathbb{R}^{d})} < \infty.
\end{equation}
By Duhamel's principle,
\begin{equation}\label{2.8}
(u(t), u_{t}(t)) = S(t)(u_{0}, u_{1}) - \int_{0}^{t} S(t - \tau)(0, |u|^{\frac{4}{d - 1}} u) d\tau.
\end{equation}
For any fixed $(u_{0}, u_{1}) \in \dot{H}^{1/2} \times \dot{H}^{-1/2}$,
\begin{equation}\label{2.9}
S(t)(u_{0}, u_{1}) \rightarrow S(T)(u_{0}, u_{1}) \qquad \text{in} \qquad \dot{H}^{1/2} \times \dot{H}^{-1/2}.
\end{equation}
Also, by $(\ref{2.7})$, Theorem $\ref{t2.1}$, and the dominated convergence theorem, for any $\epsilon > 0$ there exists $\delta(T, \epsilon) > 0$ sufficiently small such that, for any $t \in [T - \delta, T)$,
\begin{equation}\label{2.10}
\| \int_{T - \delta}^{t} S(t - \tau)(0, |u|^{\frac{4}{d - 1}} u) d\tau \|_{\dot{H}^{1/2} \times \dot{H}^{-1/2}} \lesssim \epsilon^{1 + \frac{4}{d - 1}}.
\end{equation}
As in $(\ref{2.9})$,
\begin{equation}\label{2.11}
\int_{0}^{T - \delta} S(t - \tau) (0, |u|^{\frac{4}{d - 1}} u) d\tau \rightarrow \int_{0}^{T - \delta} S(T - \tau)(0, |u|^{\frac{4}{d - 1}} u) d\tau,
\end{equation}
in $\dot{H}^{1/2} \times \dot{H}^{-1/2}$, by $(\ref{2.7})$. Therefore, $(u(t), u_{t}(t))$ converges in $\dot{H}^{1/2} \times \dot{H}^{-1/2}$ as $t \nearrow T$. By Theorem $\ref{t2.2}$, the solution to $(\ref{1.1})$ can therefore be extended past $T$.\medskip

Then if $(\ref{2.1})$ holds, set
\begin{equation}\label{2.12}
(u_{0}^{+}, u_{1}^{+}) = (u_{0}, u_{1}) - \int_{0}^{\infty} S(-\tau)(0, |u|^{\frac{4}{d - 1}} u) d\tau.
\end{equation}
Theorem $\ref{t2.1}$ and $(\ref{2.1})$ imply that $(u_{0}^{+}, u_{1}^{+}) \in \dot{H}^{1/2} \times \dot{H}^{-1/2}$ and furthermore,
\begin{equation}\label{2.13}
\lim_{t \rightarrow \infty} \| S(t)(u_{0}^{+}, u_{1}^{+}) - (u(t), u_{t}(t)) \|_{\dot{H}^{1/2} \times \dot{H}^{-1/2}} = 0.
\end{equation}
An identical argument shows that a solution to $(\ref{1.1})$ satisfying $(\ref{1.1})$ scatters backward in time. Therefore, $(\ref{2.1})$ implies that scattering holds.

The converse will not be needed here, so the proof will merely be sketched. The idea is that if scattering holds, $S(t)(u_{0}^{+}, u_{1}^{+})$ will dominate the solution for large $t$ and $S(t)(u_{0}^{-}, u_{1}^{-})$ will dominate for $|t|$ large, $t$ negative. By standard perturbative arguments combined with Strichartz estimates, this implies that there exists some $T$ sufficiently large such that
\begin{equation}\label{2.14}
\| u \|_{L_{t,x}^{\frac{2(d + 1)}{d - 1}}((-\infty, -T] \cup [T, \infty) \times \mathbb{R}^{d})} < \infty.
\end{equation}
Now then, since $u \in C_{t}^{0}(\mathbb{R}; \dot{H}^{1/2})$ and $u_{t} \in C_{t}^{0}(\mathbb{R}; \dot{H}^{-1/2})$, for any $t_{0} \in [-T, T]$, the local well-posedness result of \cite{lindblad1995existence} holds on some interval $(t_{0} - \delta(t_{0}), t_{0} + \delta(t_{0}))$, with
\begin{equation}\label{2.15}
\| u \|_{L_{t,x}^{\frac{2(d + 1)}{d - 1}}((t_{0} - \delta(t_{0}), t_{0} + \delta(t_{0})) \times \mathbb{R}^{d})} \leq 1.
\end{equation}
Since $[-T, T]$ is compact, every open cover has a finite subcover, so $(\ref{2.15})$ implies
\begin{equation}\label{2.16}
\| u \|_{L_{t,x}^{\frac{2(d + 1)}{d - 1}}([-T, T] \times \mathbb{R}^{d})} < \infty.
\end{equation}
\end{proof}

\section{The conformal energy}
Recall the scattering result of \cite{strauss1968decay}.
\begin{theorem}\label{t3.1}
Let $v$ be a solution to
\begin{equation}\label{3.1}
v_{tt} - \Delta v + |v|^{\frac{4}{d - 1}} v = 0, \qquad v(0,x) = v_{0}, \qquad v_{t}(0,x) \in v_{1},
\end{equation}
such that
\begin{equation}\label{3.2}
\| |x| \nabla_{x} v_{0} \|_{L^{2}} + \| |x| v_{1} \|_{L^{2}} + \| v_{0} \|_{L^{2}} + \int |x|^{2} |v|^{\frac{2(d + 1)}{d - 1}} dx. < \infty.
\end{equation}
Then the solution to $(\ref{3.1})$ is global and satisfies
\begin{equation}\label{3.2.1}
\| v \|_{L_{t,x}^{\frac{2(d + 1)}{d - 1}}(\mathbb{R} \times \mathbb{R}^{d})} < \infty.
\end{equation}
\end{theorem}
\begin{proof}
The proof uses the conservation of the conformal energy, see the exposition in \cite{tao2006nonlinear}, and follows directly from the lemma of \cite{klainerman2001commuting}
\begin{lemma}\label{l3.2}
The conformal energy
\begin{equation}\label{3.3}
\aligned
\mathcal E(v) = \frac{1}{4} \int_{\mathbb{R}^{d}} |(t + |x|) Lv + (d - 1) v|^{2} + |(t - |x|) \underline{L} v + (d - 1) v|^{2} \\ + 2(t^{2} + |x|^{2}) |\cancel{\nabla} v|^{2} dx + \frac{d - 1}{4(d + 1)} \int (t^{2} + |x|^{2}) |v|^{\frac{2(d + 1)}{d - 1}} dx,
\endaligned
\end{equation}
is conserved for a solution to $(\ref{3.1})$, where $L = (\partial_{t} + \frac{x}{|x|} \cdot \nabla)$ and $\underline{L} = (\partial_{t} - \frac{x}{|x|} \cdot \nabla)$.
\end{lemma}
Indeed, if Lemma $\ref{l3.2}$ is true, $E(v)(0) < \infty$ implies
\begin{equation}\label{3.4}
\| v \|_{L_{t,x}^{\frac{2(d + 1)}{d - 1}}(\mathbb{R} \times \mathbb{R}^{d})}^{\frac{2(d + 1)}{d - 1}} < \infty,
\end{equation}
which proves the Theorem.
\end{proof}
\begin{proof}[Proof of Lemma $\ref{l3.2}$.]
Define the tensors
\begin{equation}\label{3.5}
\aligned
T^{00}(t,x) &= \frac{1}{2} |\partial_{t} v|^{2} + \frac{1}{2} |\nabla v|^{2} + \frac{d - 1}{2(d + 1)} |v|^{\frac{2(d + 1)}{d - 1}}, \\
T^{0j}(t,x) &= T^{j0}(t,x) = -(\partial_{t} v)(\partial_{x_{j}} v), \\
T^{jk}(t,x) &= (\partial_{x_{j}} v \partial_{x_{k}} v) - \frac{\delta_{jk}}{2}(|\nabla v|^{2} - |\partial_{t} v|^{2}) - \delta_{jk} \frac{d - 1}{2(d + 1)} |v|^{\frac{2(d + 1)}{d - 1}}.
\endaligned
\end{equation}
The tensor functions satisfy the differential equations
\begin{equation}\label{3.6}
\partial_{t} T^{00}(t,x) + \partial_{x_{j}} T^{0j}(t,x) = 0, \qquad \partial_{t} T^{0j}(t,x) + \partial_{x_{k}} T^{jk}(t,x) = 0.
\end{equation}
The differential equations $(\ref{3.6})$ imply that the quantity
\begin{equation}\label{3.7}
Q(t) = \int (t^{2} + |x|^{2}) T^{00}(t,x) - 2t x_{j} T^{0j}(t,x) + (d - 1) t v(\partial_{t} v) - \frac{d - 1}{2} |v|^{2} dx,
\end{equation}
is conserved. The Einstein summation convention is observed. Indeed, by $(\ref{3.6})$,
\begin{equation}\label{3.8}
\aligned
\frac{d}{dt} Q(t) = 2t \int T^{00}(t,x) dx - \int (t^{2} + |x|^{2}) \partial_{x_{j}} T^{0j}(t,x) dx - 2t \int x_{j} T^{0j}(t,x) dx + 2t \int x_{j} \partial_{x_{k}} T^{jk}(t,x) dx \\
+ (d - 1) \int v (\partial_{t} v) dx + (d - 1) t \int (\partial_{t} v)^{2} dx + (d - 1) t \int v(\Delta v - |v|^{\frac{4}{d - 1}} v) dx - (d - 1) \int v (\partial_{t} v) dx.
\endaligned
\end{equation}
Integrating the second term in $(\ref{3.8})$ by parts,
\begin{equation}\label{3.9}
\aligned
=  2t \int T^{00}(t,x) dx - 2t \int \delta_{jk} T^{jk}(t,x) dx  + (d - 1) t \int (\partial_{t} v)^{2} dx + (d - 1) t \int v(\Delta v - |v|^{\frac{4}{d - 1}} v) dx.
\endaligned
\end{equation}
Since $\delta_{jk} \delta_{jk} = d$,
\begin{equation}\label{3.10}
\aligned
= 2t \int T^{00}(t,x) dx - 2t \int |\nabla v|^{2} + dt \int (|\nabla v|^{2} - |\partial_{t} v|^{2}) dx + \frac{d(d - 1)t}{d + 1} \int |v|^{\frac{2(d + 1)}{d - 1}} dx \\ + (d - 1)t \int |\partial_{t} v|^{2} dx - (d - 1)t \int |\nabla v|^{2} dx - (d - 1)t \int |v|^{\frac{2(d + 1)}{d - 1}} dx.
\endaligned
\end{equation}
Doing some algebra,
\begin{equation}\label{3.11}
= 2t \int T^{00}(t,x) dx - t \int |\nabla v|^{2} dx - t \int |\partial_{t} v|^{2} dx - \frac{d - 1}{d + 1} t \int |v|^{\frac{2(d + 1)}{d - 1}} dx = 0.
\end{equation}
Therefore, $Q(t)$ is conserved.

Now then,
\begin{equation}\label{3.12}
\aligned
\int (t^{2} + |x|^{2}) T^{00}(t,x) dx - \int 2t x_{j} T^{0j}(t,x) dx \\ = \int (t^{2} + |x|^{2}) (\frac{1}{2} |\partial_{t} v|^{2} + \frac{1}{2} |\frac{x}{|x|} \cdot \nabla v|^{2} + \frac{1}{2} |\cancel{\nabla} v|^{2} + \frac{d - 1}{2(d + 1)} |v|^{\frac{2(d + 1)}{d - 1}}) dx \\
= \frac{1}{4} \int (t + |x|)^{2} |Lv|^{2} dx + \frac{1}{4} \int (t - |x|)^{2} |\underline{L}v|^{2} + \frac{1}{2} (t^{2} + |x|^{2}) |\cancel{\nabla} v|^{2} dx + \frac{d - 1}{2(d + 1)} \int (t^{2} + |x|^{2}) |v|^{\frac{2(d + 1)}{d - 1}} dx.
\endaligned
\end{equation}
Next, integrating by parts,
\begin{equation}\label{3.13}
\aligned
\frac{1}{2} \langle (t + |x|) Lv, (d - 1) v \rangle_{L^{2}} + \frac{1}{2} \langle (t - |x|) Lv, (d - 1) v \rangle_{L^{2}} \\ = (d - 1)t \int (\partial_{t} v) v dx + (d - 1) \int v (x \cdot \nabla v) dx = (d - 1)t \int (\partial_{t} v) v dx - \frac{d(d - 1)}{2} \int |v|^{2} dx.
\endaligned
\end{equation}
Since
\begin{equation}\label{3.14}
-\frac{d(d - 1)}{2} \int |v|^{2} dx + \frac{(d - 1)^{2}}{2} \int |v|^{2} dx = -\frac{d - 1}{2} \int |v|^{2} dx,
\end{equation}
$(\ref{3.12})$--$(\ref{3.14})$ imply that $Q(t)$ is equal to the right hand side of $(\ref{3.3})$.
\end{proof}

\section{Proof of Theorem $\ref{t1.1}$}
\begin{proof}[Proof of Theorem $\ref{t1.1}$]
Since $B_{1,1}^{\frac{d}{2} + \frac{1}{2}} \subset \dot{H}^{1/2}$ and $B_{1,1}^{\frac{d}{2} - \frac{1}{2}} \subset \dot{H}^{-1/2}$, $(\ref{1.1})$ is locally well-posed on some interval $(-T_{1}, T_{2})$, where $T_{1}, T_{2} > 0$. To prove global well-posedness and scattering, it suffices to to show that
\begin{equation}\label{4.1}
\| u \|_{L_{t,x}^{\frac{2(d + 1)}{d - 1}}([0, T_{2}) \times \mathbb{R}^{d})} < \infty, \qquad \text{and} \qquad \| u \|_{L_{t,x}^{\frac{2(d + 1)}{d - 1}}((-T_{1}, 0] \times \mathbb{R}^{d})} < \infty,
\end{equation}
where $(-T_{1}, T_{2})$ is the maximal interval of existence of the solution to $(\ref{1.1})$, and that the bound does not depend on $T_{1}$ and $T_{2}$.\medskip

To prove this, decompose the solution on $(-T_{1}, T_{2})$, $u = v + w$, where $v$ and $w$ solve the equations
\begin{equation}\label{4.2}
v_{tt} - \Delta v + |v + w|^{\frac{4}{d - 1}} (v + w) = 0, \qquad v(0,x) = 0, \qquad v_{t}(0,x) = 0,
\end{equation}
and
\begin{equation}\label{4.3}
w_{tt} - \Delta w = 0, \qquad w(0,x) = u_{0}, \qquad w_{t}(0,x) = u_{1}.
\end{equation}
Strichartz estimates, Theorem $\ref{t2.1}$, imply $\| w \|_{L_{t,x}^{\frac{2(d + 1)}{d - 1}}(\mathbb{R} \times \mathbb{R}^{d})} \lesssim_{d} A$, so to prove $(\ref{2.1})$, it suffices to show $\| v \|_{L_{t,x}^{\frac{2(d + 1)}{d - 1}}(\mathbb{R} \times \mathbb{R}^{d})} < \infty$.

Split
\begin{equation}\label{4.4}
|v + w|^{\frac{4}{d - 1}} (v + w) = |v|^{\frac{4}{d - 1}} v + F, \qquad \text{where} \qquad |F| \lesssim |v|^{\frac{4}{d - 1}} |w| + |w|^{\frac{d + 3}{d - 1}}.
\end{equation}
Now let $\mathcal E(t)$ denote the conformal energy of $v$, see $(\ref{3.3})$. Then by the computations proving Lemma $\ref{l3.2}$,
\begin{equation}\label{4.5}
\frac{d}{dt} \mathcal E(t) = -\frac{1}{2} \langle (t + |x|) Lv + (d - 1) v, (t + |x|) F \rangle - \frac{1}{2} \langle (t - |x|) \underline{L} v + (d - 1) v, (t - |x|) F \rangle.
\end{equation}
Since
\begin{equation}\label{4.6}
\| (t + |x|) Lv + (d - 1) v \|_{L^{2}} + \| (t - |x|) \underline{L} v + (d - 1) v \|_{L^{2}} \lesssim \mathcal E(t)^{1/2},
\end{equation}
\begin{equation}\label{4.7}
\frac{d}{dt} \mathcal E(t) \lesssim \mathcal E(t)^{1/2} (\| t |w|^{\frac{2}{d - 1}} \|_{L^{\infty}} + \| |x| |w|^{\frac{2}{d - 1}} \|_{L^{\infty}}) \| |w|^{\frac{d - 3}{d - 1}} |v|^{\frac{4}{d - 1}} + |w|^{\frac{d + 1}{d - 1}} \|_{L^{2}}.
\end{equation}
The dispersive estimate implies
\begin{equation}\label{4.8}
\| t |w|^{\frac{2}{d - 1}} \|_{L^{\infty}} \lesssim A^{\frac{2}{d - 1}},
\end{equation}
and the radial Sobolev embedding theorem implies
\begin{equation}\label{4.9}
\| |x| |w|^{\frac{2}{d - 1}} \|_{L^{\infty}} \lesssim A^{\frac{2}{d - 1}}.
\end{equation}
Therefore,
\begin{equation}\label{4.10}
\frac{d}{dt} \mathcal E(t) \lesssim A^{\frac{2}{d - 1}} \mathcal E(t)^{1/2} (\| w \|_{L_{x}^{\frac{2(d + 1)}{d - 1}}}^{\frac{d + 1}{d - 1}} + \| v \|_{L_{x}^{\frac{2(d + 1)}{d - 1}}}^{\frac{4}{d - 1}} \| w \|_{L_{x}^{\frac{2(d + 1)}{d - 1}}}^{\frac{d - 3}{d - 1}}).
\end{equation}
Then by $(\ref{3.3})$,
\begin{equation}\label{4.11}
\frac{d}{dt} \mathcal E(t) \lesssim A^{\frac{2}{d - 1}} \mathcal E(t)^{1/2} \| w \|_{L_{x}^{\frac{2(d + 1)}{d - 1}}}^{\frac{d + 1}{d - 1}} + \frac{1}{t^{\frac{4}{d + 1}}} A^{\frac{2}{d - 1}} \mathcal E(t)^{\frac{1}{2} + \frac{2}{d + 1}} \| w \|_{L_{x}^{\frac{2(d + 1)}{d - 1}}}^{\frac{d - 3}{d - 1}}.
\end{equation}

By $(\ref{3.3})$, $\mathcal E(0) = 0$, the fundamental theorem of calculus, and $(\ref{4.11})$,
\begin{equation}\label{4.12}
\| v \|_{L_{t,x}^{\frac{2(d + 1)}{d - 1}}([0, T_{2}) \times \mathbb{R}^{d})}^{\frac{2(d + 1)}{d - 1}} \lesssim \int_{0}^{T_{2}} \frac{1}{t^{2}} \mathcal E(t) dt \lesssim \int_{0}^{T_{2}} \frac{1}{t^{2}} \int_{0}^{t} A^{\frac{2}{d - 1}} \mathcal E(\tau)^{1/2} \| w \|_{L_{x}^{\frac{2(d + 1)}{d - 1}}}^{\frac{d + 1}{d - 1}} + \frac{1}{\tau^{\frac{4}{d + 1}}} A^{\frac{2}{d - 1}} \mathcal E(\tau)^{\frac{1}{2} + \frac{2}{d + 1}} \| w \|_{L_{x}^{\frac{2(d + 1)}{d - 1}}}^{\frac{d - 3}{d - 1}} d\tau.
\end{equation}
Then by Fubini's theorem,
\begin{equation}\label{4.13}
(\ref{4.12}) = A^{\frac{2}{d - 1}} \int_{0}^{T_{2}} \mathcal E(\tau)^{\frac{1}{2}} \| w(\tau) \|_{L_{x}^{\frac{2(d + 1)}{d - 1}}}^{\frac{d + 1}{d - 1}} \int_{\tau}^{T_{2}} \frac{1}{t^{2}} dt + A^{\frac{2}{d - 1}} \int_{0}^{T_{2}} \frac{1}{\tau^{\frac{4}{d + 1}}} \mathcal E(\tau)^{\frac{1}{2} + \frac{2}{d + 1}} \int_{\tau}^{T_{2}} \frac{1}{t^{2}} dt.
\end{equation}
\begin{equation}\label{4.14}
\lesssim A^{\frac{2}{d - 1}} \int_{0}^{T_{2}} \frac{1}{\tau} \mathcal E(\tau)^{\frac{1}{2}} \| w(\tau) \|_{L_{x}^{\frac{2(d + 1)}{d - 1}}}^{\frac{d + 1}{d - 1}} d\tau + A^{\frac{2}{d - 1}} \int_{0}^{T_{2}} \frac{1}{\tau^{\frac{d + 5}{d + 1}}} \mathcal E(\tau)^{\frac{d + 5}{2(d + 1)}} \| w(\tau) \|_{L_{x}^{\frac{2(d + 1)}{d - 1}}}^{\frac{d - 3}{d - 1}} d\tau.
\end{equation}
By H{\"o}lder's inequality,
\begin{equation}\label{4.15}
\lesssim A^{\frac{2}{d - 1}} (\int_{0}^{T_{2}} \frac{1}{\tau^{2}} \mathcal E(\tau) d\tau)^{1/2} (\int_{0}^{T_{2}} \| w(\tau) \|_{L_{x}^{\frac{2(d + 1)}{d - 1}}}^{\frac{2(d + 1)}{d - 1}} d\tau)^{1/2} + A^{\frac{2}{d - 1}} (\int_{0}^{T_{2}} \frac{1}{\tau^{2}} \mathcal E(\tau) d\tau)^{\frac{d + 5}{2(d + 1)}} (\int_{0}^{T_{2}} \| w(\tau) \|_{L_{x}^{\frac{2(d + 1)}{d - 1}}}^{\frac{2(d + 1)}{d - 1}} d\tau)^{\frac{d - 3}{2(d + 1)}}.
\end{equation}

By Strichartz estimates,
\begin{equation}\label{4.16}
\int_{\mathbb{R}} \| w(t) \|_{L_{x}^{\frac{2(d + 1)}{d - 1}}}^{\frac{2(d + 1)}{d - 1}} dt \lesssim A^{\frac{2(d + 1)}{d - 1}}.
\end{equation}
Therefore, doing some algebra,
\begin{equation}\label{4.17}
\| v \|_{L_{t,x}^{\frac{2(d + 1)}{d - 1}}([0, T_{2}) \times \mathbb{R}^{d})}^{\frac{2(d + 1)}{d - 1}} \lesssim \int_{0}^{T_{2}} \frac{1}{t^{2}} \mathcal E(t) dt \lesssim A^{\frac{2(d + 3)}{d - 1}} + A^{\frac{2(d + 1)}{d - 3}}.
\end{equation}
Thus, $T_{2} = \infty$ and the solution to $(\ref{1.1})$ is global forward in time. By time reversal symmetry, the proof is complete.
\end{proof}

\begin{remark}
To place the use of Fubini's theorem on firm footing, it is possible to approximate $u_{0}$ and $u_{1}$ by smooth, compactly supported functions. Then by Theorem $\ref{t3.1}$, $\| u \|_{L_{t,x}^{\frac{2(d + 1)}{d - 1}}(\mathbb{R} \times \mathbb{R}^{d})} < \infty$, and the bounds are controlled by the right hand side of $(\ref{4.17})$.
\end{remark}

\section*{Acknowledgements}
During the time of writing this paper, the author was partially supported by NSF Grant DMS-$1764358$. The author is also grateful to Andrew Lawrie and Walter Strauss for some helpful conversations at MIT on subcritical nonlinear wave equations.

\bibliography{biblio}

\newcommand{\etalchar}[1]{$^{#1}$}
\begin{thebibliography}{Dod18b}

\bibitem[Dod18a]{dodson2018global}
Benjamin Dodson.
\newblock Global well-posedness and scattering for the radial, defocusing,
  cubic nonlinear wave equation.
\newblock {\em arXiv preprint arXiv:1809.08284}, 2018.

\bibitem[Dod18b]{dodson2018globalAPDE}
Benjamin Dodson.
\newblock Global well-posedness and scattering for the radial, defocusing,
  cubic wave equation with initial data in a critical {B}esov space.
\newblock {\em Analysis \& PDE}, 12(4):1023--1048, 2018.

\bibitem[Dod18c]{dodson2018global2}
Benjamin Dodson.
\newblock Global well-posedness for the radial, defocusing, nonlinear wave
  equation for $3< p< 5$.
\newblock {\em arXiv preprint arXiv:1810.02879}, 2018.

\bibitem[Dod21]{dodson2021scattering}
Benjamin Dodson.
\newblock Scattering for the defocusing, cubic nonlinear schr $\{$$\backslash$"
  o$\}$ dinger equation with initial data in a critical space.
\newblock {\em arXiv preprint arXiv:2110.06987}, 2021.

\bibitem[GV95]{ginibre1995generalized}
Jean Ginibre and Giorgio Velo.
\newblock Generalized strichartz inequalities for the wave equation.
\newblock {\em Journal of functional analysis}, 133(1):50--68, 1995.

\bibitem[K{\etalchar{+}}94]{kato1994lq}
Tosio Kato et~al.
\newblock An lq, r-theory for nonlinear schr{\"o}dinger equations.
\newblock {\em Spectral and scattering theory and applications}, 23:223--238,
  1994.

\bibitem[Kap89]{kapitanski1989some}
Lev~Vil'evich Kapitanski.
\newblock Some generalizations of the strichartz--brenner inequality.
\newblock {\em Algebra i Analiz}, 1(3):127--159, 1989.

\bibitem[Kla01]{klainerman2001commuting}
Sergiu Klainerman.
\newblock A commuting vectorfields approach to strichartz-type inequalities and
  applications to quasi-linear wave equations.
\newblock {\em International Mathematics Research Notices}, 2001(5):221--274,
  2001.

\bibitem[KT98]{keel1998endpoint}
Markus Keel and Terence Tao.
\newblock Endpoint strichartz estimates.
\newblock {\em American Journal of Mathematics}, 120(5):955--980, 1998.

\bibitem[LS95]{lindblad1995existence}
Hans Lindblad and Christopher~D Sogge.
\newblock On existence and scattering with minimal regularity for semilinear
  wave equations.
\newblock {\em Journal of Functional Analysis}, 130(2):357--426, 1995.

\bibitem[MYZ20]{miao2020global}
Changxing Miao, Jianwei Yang, and Tengfei Zhao.
\newblock The global well-posedness and scattering for the 5-dimensional
  defocusing conformal invariant nlw with radial initial data in a critical
  besov space.
\newblock {\em Pacific Journal of Mathematics}, 305(1):251--290, 2020.

\bibitem[Sog95]{sogge1995lectures}
Christopher~Donald Sogge.
\newblock {\em Lectures on non-linear wave equations}, volume~2.
\newblock International Press Boston, MA, 1995.

\bibitem[SS00]{shatah2000geometric}
Jalal M~Ihsan Shatah and Michael Struwe.
\newblock {\em Geometric wave equations}, volume~2.
\newblock American Mathematical Soc., 2000.

\bibitem[Str68]{strauss1968decay}
Walter~A Strauss.
\newblock Decay and asymptotics for box u= f (u).
\newblock {\em Journal of Functional Analysis}, 2(4):409--457, 1968.

\bibitem[Tao06]{tao2006nonlinear}
Terence Tao.
\newblock {\em Nonlinear dispersive equations: local and global analysis}.
\newblock Number 106. American Mathematical Soc., 2006.

\end{thebibliography}
\bibliographystyle{alpha}

\end{document}